\documentclass[a4paper,11pt]{article}

\usepackage[dvipsnames]{xcolor}
\usepackage[english]{babel}
\usepackage{amssymb, enumerate}
\usepackage{latexsym}
\usepackage[all,cmtip]{xy}
\usepackage{amsmath,amssymb,amsthm}
\usepackage{comment}
\usepackage{tikz}
\usepackage{marvosym}
\usepackage{tabu}

\usepackage{color}

\usepackage{tikz}
\usepackage{tikz-cd}
\usepackage{tkz-euclide}
\usepackage{caption}
\usepackage{enumerate}
\usepackage[all]{xy}
\usepackage{MnSymbol}
\usepackage[margin=1.59in]{geometry}
\usepackage{float}
\usepackage{thmtools}
\usepackage{multicol}
\restylefloat{table}
\usepackage{footmisc}
\renewcommand{\thefootnote}{\fnsymbol{footnote}}
\usetikzlibrary{shapes.geometric,decorations}
\usepackage{titlesec}%\bfseries
\usepackage{listings}
\usepackage{relsize}
\definecolor{burgundy}{rgb}{0.5, 0.0, 0.13}
\definecolor{asparagus}{rgb}{0.163, 0.66, 0.42}
\usepackage[colorlinks=true, citecolor=Blue, linkcolor=Blue, urlcolor=Blue, linktocpage=true, pdfpagelabels, ]{hyperref}
\usepackage{cleveref}

\newtheorem{theorem}{Theorem}[section]
\newtheorem*{theorem*}{Theorem}

\newtheorem{lemma}[theorem]{Lemma}
\newtheorem{corollary}[theorem]{Corollary}
\newtheorem{proposition}[theorem]{Proposition}
\newtheorem{remark}[theorem]{Remark}

\newcommand\blfootnote[1]{%
	\begingroup
	\renewcommand\thefootnote{}\footnote{#1}%
	\addtocounter{footnote}{-1}%
	\endgroup
}
\begin{document}\setcounter{MaxMatrixCols}{50}
\title{$2$-Blocks whose defect group is homocyclic and whose inertial quotient contains a Singer cycle}
\author{Elliot Mckernon$^{\dagger}$}
\date{\today}
\maketitle
\begin{abstract}
We consider a block $B$ of a finite group with defect group $D \cong (C_{2^m})^n$ and inertial quotient $\mathbb{E}$ containing a Singer cycle (an element of order $2^n-1$). This implies $\mathbb{E} = E \rtimes F$, where $E \cong C_{2^n-1}$, $F \leq C_n$, and $E$ acts transitively on the elements in $D$ of order $2$, and freely on $D \backslash \{1\}$. We classify the basic Morita equivalence classes of $B$ over a complete discrete valuation ring $\mathcal{O}$: when $m=1$, $B$ is basic Morita equivalent to the principal block of one of $SL_2(2^n) \rtimes F$, $D \rtimes \mathbb{E}$, or $J_1$ (where $J_1$ occurs only when $n=3$). When $m>1$, $B$ is basic Morita equivalent to $D \rtimes \mathbb{E}$.

\blfootnote{${}^{\dagger}$ School of Mathematics, University of Manchester, Manchester, M13 9PL, United Kingdom.\\ Email:
elliot.mckernon@manchester.ac.uk.}
%2010 Mathematics Subject Classification: Primary $20$C$20$; Secondary $16$D$90$, $20$C$05$.

Keywords: Block theory; modular representation theory, finite groups; Donovan's conjecture; Singer cycles; Morita equivalence;

%2010 Mathematics Subject Classification: Primary 20A10; Secondary 46D20,423A8.
\end{abstract}
%%%%%%%%%%%%%%%%%%%%%%%%%%%%%%%

\section{Introduction}
Given a finite group $G$ and a field $k$ of prime characteristic $p$ such that $p \divides |G|$, the modular representation theory of $G$ over $k$ reduces to studying summands of the group algebra, $kG$. These are called blocks, and to each block $B$ we associate a $p$-subgroup of $G$ that controls its structure - the defect group. Donovan's conjecture states that, given a $p$-group $D$, there are finitely many blocks with defect group $D$, up to Morita equivalence. 

Donovan's conjecture has been verified in several cases: for example, the main result of \cite{ekks} tells us there are only finitely many Morita equivalence classes of blocks with defect group $(C_2)^n$, for $n \in \mathbb{N}$. This inspired a program of explicitly describing these classes, which has been achieved for $n = 2,3, 4$, and $5$,  in \cite{kleinfour}, \cite{eaton3}, \cite{eaton4}, and \cite{CesC2^5}, respectively. These results and more can be found on the online block library \cite{blockwiki}. The number of classes rises as the rank of $D$ increases, from three when $D \cong (C_2)^2$, to $34$ when $D \cong (C_2)^5$. Classifying blocks when the rank is arbitrary is not realistic, currently: in order to relax our constraints on the rank, we'll need to focus our attention on special classes of block. 

Let $B_D$ be a Brauer correspondent of $B$ in $DC_G(D)$. We write $N_G(D,B_D)$ for the elements in $G$ that stabilise $D$ and $B_D$ under conjugation. The inertial quotient of $B$ is then defined as $N_G(D,B_D)/C_G(D)$. This is a $p'$-group that embeds into $\operatorname{Out}(D)$, and contextualises the structural information that $D$ provides. 

There is a pattern in the $2$-blocks with small elementary abelian defect groups: considering only blocks whose inertial quotient contains an element of order $2^n-1$, the number of Morita equivalence classes stays low even as $n$ increases. Elements of order $2^n-1$ in \linebreak$\operatorname{Out}({C_2}^n) \cong GL_n(2)$ are called Singer cycles, and the subgroups they generate act transitively on the non-trivial elements of $(C_2)^n$. This pattern can also be found when the defect group is homocyclic, and it renders such blocks amenable to analysis. In \Cref{main}, we classify these blocks. 

Throughout, let $k$ be an algebraically closed field of prime charasteristic $2$, and $\mathcal{O}$ a complete discrete valuation ring with residue field $k$ and field of fractions $K$ of characteristic $0$. We assume $K$ is large enough for the groups under consideration. Let $G$ be a finite group such that $2 \divides |G|$, and consider a block $B$ of $\mathcal{O}G$. 

\begin{theorem} \label{main} 
Let $G$ be a finite group, and $B$ be a block of $\mathcal{O}G$ with defect group $D \cong (C_{2^m})^n$,  and inertial quotient $\mathbb{E}$ containing an element of order $2^n-1$. Then $\mathbb{E} \cong E \rtimes F$ where $E \cong C_{2^n-1}$ and $F \leq C_n$, and $B$ is basic Morita equivalent to the principal block of one of the following: \begin{enumerate} \item $D \rtimes (E \rtimes F)$ \label{case1}
		\item $SL_2(2^n) \rtimes F$  \label{case2}
		\item $J_1$  \label{case3}
\end{enumerate} In particular, if $n>3$, then $B$ lies in case (\ref{case1}) or (\ref{case2}), and if $m>1$ then $B$ lies in case (\ref{case1}). \end{theorem}

To prove \Cref{main} we will study the subgroup structure of $G$, and show it contains a unique quasisimple component. \cite[Theorem 6.1]{ekks} provides a list of candidates for the structure of this quasisimple and its blocks, and we'll compare these candidates to what we know about $G$ and $B$.

First, in order to understand the action of the inertial quotient, we'll study Singer cycles in \Cref{section Singer}. We'll then collect several standard block-theoretic results in \Cref{section others}, as well as tools more specific to our case in \Cref{section subgroup}. In \Cref{section proof} we will apply these results to prove \Cref{main}. 

\section{Singer Cycles} \label{section Singer}
The key insight into the block considered in \Cref{main} is the action of its inertial quotient on its defect group. In this section we describe and explain that action. 

Let $p$ be a prime and consider the finite field of order $p^n$, $\mathbb{F}_{p^n}$. Recall that the additive group of $\mathbb{F}_{p^n}$ is isomorphic to $(C_p)^n$, and the multiplicative group isomorphic to $ C_{p^n-1}$, where the latter acts by multiplication on the former. This action is transitive on the non-trivial elements. More generally, if $P:= (C_p)^n$ and $G:=\operatorname{Aut}(P) \cong GL_n(p)$, then every element of $G$ of order $p^n-1$ generates a cyclic group which acts transitively on $P \backslash \{1\}$. These elements of order $p^n-1$ in $G$ are called \textit{Singer cycles} of $G$, and they are well-understood:

\begin{proposition} \label{Sings} Let $n$ be a natural number and $p$ be a prime. 
	\begin{enumerate}
		\item $GL_n(p)$ contains a Singer cycle for all $p$ and $n$. 
		\item The Singer cycles in $GL_n(p)$ are the elements of maximal order. 
		\item The subgroups of $GL_n(p)$ that are generated by a Singer cycle are conjugate.
		\item The subgroups of $GL_n(p)$ that are generated by a Singer cycle act regularly and irreducibly on non-zero vectors of the underlying vector space $V$. \label{Singsregular}
		\item Let $X$ be a subgroup of $GL_n(p)$ generated by a Singer cycle. Then $C_{GL_n(p)}(X)=X$ and $N_{GL_n(p)}(X)/X \cong C_n$. \label{intro singer prop normaliser}
	\end{enumerate}
\end{proposition}
\begin{proof}
$(1)$ and $(4)$ are proven in \cite[2.1]{singerovergroups}, $(2)$ in \cite[3.27]{muller}, and $(3)$ and $(5)$ in \cite[p.187]{huppert2013endliche}. \end{proof}

\begin{theorem} \label{intro singer kantor}
	(\cite{kantor}) If $H$ is a subgroup of $GL_n(q)$ that contains a Singer cycle, then there is a natural embedding of $GL_{n/s}(q^s)$ as a normal subgroup of $H$, for some $s \divides n$. 
\end{theorem}

\begin{corollary} \label{intro singer maximal sub}
Let $q=p^r$ for $p$ a prime. Let $G \cong GL_n(q)$, and $H$ be a $p'$-subgroup of $G$ containing an element of order $q^n-1$. If $H$ is maximal with respect to containment, satisfying these properties, then $H \cong C_{q^n-1} \rtimes C_{n'}$, where $n'$ is the largest divisor of $n$ not divisible by $p$.
\end{corollary}
\begin{proof}
 \Cref{intro singer kantor} tells us there is an embedding of $GL_{n/s}(q^s)$ in $H$ with normal image, for some divisor $s$ of $n$.  If $s<n$, then $|GL_{n/s}(q^s)|=(q^n-1)(q^n-q^s)...(q^n-q^{(n-1)s})$. This implies $p  \divides |H|$, contradicting our assumption. Thus we must have  $n=s$, and so $GL_1(q^n) \cong C_{q^n-1} \lhd H$, implying $H \lhd N_{GL_n(q)}(C_{q^n-1})$. \Cref{Sings} (\ref{intro singer prop normaliser}) then tells us $H \cong C_{q^n-1} \rtimes C_{n'}$, where $n'$ is the largest $p'$ divisor of $n$.
\end{proof}

\begin{remark} \label{homo} We can extend this analysis to homocyclic groups: let $Q=(C_{p^m})^n$, and recall that $\Omega(Q)$ denotes the elements of order $p$ in $Q$, so that $\Omega(Q)=(C_p)^n$.  Consider the map $\varphi: \operatorname{Aut}(Q) \to \operatorname{Aut}(\Omega(Q))$ where $\varphi(\alpha)$ is the automorphism of $\Omega{(Q)}$ induced by $\alpha$. This is a surjective homomorphism, and if $\alpha \in \ker \varphi$ then the order of $\alpha$ is a power of $p$, by \cite[Theorem 1.15]{PRIMEPOWER}, implying $\ker(\varphi)=O_p(\operatorname{Aut}(Q))$ and $\operatorname{Aut}(Q)/O_p(\operatorname{Aut}(Q)) \cong GL_n(p)$. Thus \Cref{Sings} tells us that $\operatorname{Aut}(Q)$ contains a unique conjugacy class of elements of order $p^n-1$, and these elements generate groups which act transitively on $\Omega(Q) \backslash \{1\}$. An automorphism of $Q$ is determined by its action on the generators, and so there must be a set of generators of $Q$ that are transitively permuted by a Singer cycle subgroup. In particular, cyclic subgroups of $\operatorname{Aut}(Q)$ of order $p^n-1$ must act freely on the non-trivial elements of $Q$. \end{remark}

%if  $Q= (C_{p^m})^n$  then $\operatorname{Aut}(Q) \cong GL_n(\mathbb{Z}/p^m\mathbb{Z})$, and if we take a maximal ideal $M$ of $R=\mathbb{Z}/p^m\mathbb{Z}$, we have a sequence of homomorphisms $GL_n(R)=GL_n(R/M^m) \to GL_n(R/M^{m-1}) \to ... \to GL_n(R/M)=GL_n(p)$, where the kernel of each morphism consists of matrices congruent to the identity modulo the relevant power of $M$. This implies $|GL_n(R)|=p^{n^2(m-1)}|GL_n(p)|$, and, further, it is known that $GL_n(R)=O_p(GL_n(R)) \rtimes GL_n(p)$. Thus we will refer to elements of order $p^n-1$ in $GL_n(R)$ as Singer cycles. We summarise below, recalling that $\Omega(Q)$ denotes the subgroup of elements of order $p$ in $Q$.

\begin{lemma} \label{homo Sing lem} Let $Q \cong (C_{p^m})^n$ and $\mathbb{E} \leq \operatorname{Aut}(Q)$ such that $\mathbb{E}$ has $p'$ order and contains a Singer cycle. Then the following hold: \begin{enumerate}
\item $\mathbb{E} \cong E \rtimes F$ where $E \cong C_{p^n-1}$ and $F \leq F' \cong C_n$. \label{'1}
\item If $f$ generates $F'$, then $F$ fixes a subgroup $(C_p)^r \leq \Omega(Q)$ if and only if $F=\langle f^r \rangle$. \label{'2}
\item If $H$ is a non-trivial normal subgroup of $\mathbb{E}$, then $C_Q(H)=1$. \label{'3}
\end{enumerate}    \end{lemma}
\begin{proof}
\Cref{intro singer maximal sub} tells us $C_{p^n-1} \leq \mathbb{E} \leq C_{p^n-1} \rtimes F'$, where $F' \cong C_n$, proving (\ref{'1}). Further, $F'$ corresponds to the field automorphisms of $\mathbb{F}_{p^n}$, and thus is generated by the Frobenius endomorphism $\varphi:x \mapsto x^p$. The fixed points of $\mathbb{F}_{p^n}^+ \cong \Omega(Q)$ correspond to subfields of $\mathbb{F}_{p^n}$, and thus $F$ fixes $(C_p)^r \leq \Omega(Q)$ if and only if $F$ is generated by $\varphi^r$, proving (\ref{'2}).

To prove (\ref{'3}), let $H \lhd \mathbb{E}$, and let $q \in Q$ such that $H$ fixes $q$, i.e. $H \leq Stab_{\mathbb{E}}(q)$. Since $H$ is normal, for all $x \in \mathbb{E}$ we have $H={}^xH \leq Stab_{\mathbb{E}}({}^xq)$. Thus $H$ fixes the orbit of $q$ under the action of $\mathbb{E}$. Therefore either $q=1$ or $H$ centralises all of $Q$. The latter is a contradiction, since $H$ contains non-trivial automorphisms. Thus $C_Q(H)=1$. \end{proof}

\section{Collecting Tools} \label{section others}
We can study a block $B$ of $G$ by comparing it to blocks of subgroups of $G$. If $B$ has defect group $D$, the Brauer correspondence relates $B$ to blocks of subgroups $H$ satisfying $DC_G(D) \leq H \leq N_G(D)$. This correspondence lies ``at the very heart of block theory'' \cite{alperincorre}: the defect group is defined as the largest $p$-subgroup $D$ such that $B$ has a Brauer correspondent in $DC_G(D)$, and Brauer's first main theorem tells us that the correspondence defines a bijection between blocks of $G$ with defect group $D$ and blocks of $N_G(D)$ with defect group $D$.

% If $P$ is a $p$-subgroup of $G$, and $H$ is a subgroup such that $PC_G(P) \leq H \leq N_G(P)$, then the Brauer homomorphism with respect to $P$ is the restriction to $Z(\mathcal{O}G)$ of the linear projection map, denoted $\operatorname{Br}^G_P:Z(\mathcal{O}G) \to Z(\mathcal{O}H)$, which sends an element $g \in G$ to $g$ if $g \in C_G(P)$, and to $0$ otherwise. It is a ring homomorphism, and thus sends the block idempotent of $B$ to an idempotent in $\mathcal{O}H$. If this idempotent is non-zero, we can decompose it into block idempotents, and we say that the corresponding blocks of $\mathcal{O}H$ are Brauer correspondents of $B$. When $b$ is a Brauer correspondent of $B$, this is denoted as $b^G=B$. Brauer's first main theorem tells us the Brauer correspondence defines a bijection between blocks of $\mathcal{O}G$ with defect group $P$ and blocks of $N_G(P)$ with defect group $P$.

We can also relate blocks of $G$ to blocks of normal subgroups: when $N \lhd G$ and $b$ is a block of $N$, we say $B$ \textbf{covers} $b$ if $Bb \neq 0$. The comparison of $B$ and $b$ is called Clifford theory. 

\begin{proposition} (\cite[15.1]{alperin}) \label{alp15} Let $N \lhd G$ be finite groups, and let $B$ be a $p$-block of $\mathcal{O}G$ covering a block $b$ of $\mathcal{O}N$. Then the following hold: \begin{enumerate} \item The blocks of $N$ covered by $B$ form a $G$-conjugacy class. 
		\item Each defect group of $b$ is the intersection of a defect group of $B$ with $N$. \label{alpintersect}
%		\item There is a defect group of $B$ contained in $N_G(b)$. 
		\item There is a block $B'$ of $G$ covering $b$ such that $B'$ has a defect group $D'$ satisfying 
		
		{\centering $ [D':D' \cap N]=[\operatorname{Stab}_G(b):N]_p.$ \label{alpblockabove}\par} 
		
		\item If the centraliser in $G$ of a defect group of $b$ is contained in $N$ then $B=b^G$ and $B$ is the only block of $G$ covering $b$. \label{alp CG(D)<N}	\end{enumerate}\end{proposition}
	
%		containing a conjugate of a defect group of every block of $G$ covering $b$ and

Considering the prominent role of the defect group in \Cref{alp15}, it is no surprise that we can extract more information about the relationship between $B$ and $b$ when the index of $N$ is a power of $p$:

\begin{proposition} \label{index p}
	Let $N \lhd G$ be finite groups such that $[G:N]$ is a power of a prime $p$. Let $B$ be a $p$-block of $G$ covering a block $b$ of $N$. Then $B$ is the unique block of $G$ covering $b$, and if $b$ is $G$-stable then $B$ and $b$ share a block idempotent. \end{proposition}
\begin{proof}
	\cite[V, Lemma 3.5.]{feit} tells us $B$ is the unique block covering $b$, and \cite[Proposition 6.8.11]{linckelmann} tells us that $B$ and $b$ share a block idempotent when $b$ is $G$-stable.
\end{proof}

When $G/N$ is solvable we can relate the index of $N$ in $G$ to the index of $D \cap N$ in $D$, as described in \Cref{solv quotient}. In particular, when $G/N$ is solvable and $D \leq N$, we will be able to assume that $p \not \divides [G:N]$, in which case \Cref{zhou} permits further comparison of $B$ and $b$. Recall that a block of $G$ is\textit{ quasiprimitive} if every block it covers is $G$-stable. 

\begin{lemma} {(\cite[2.4]{CesC2^5})} \label{solv quotient}
	Let $N \lhd G$ be finite groups such that $G/N$ is solvable, and $B$ a quasiprimitive block of $G$ with abelian defect group $D$. Then $DN/N$ is a Sylow-$p$-subgroup of $G/N$. 
\end{lemma}

Following \cite{linckelmann}, two algebras, $A$ and $B$, are Morita equivalent if and only if there is an $(A,B)$-bimodule $M$ and a $(B,A)$-bimodule $N$ such that $M \otimes_B N \cong A$ as $(A,A)$-bimodules, and $N \otimes_A M \cong B$ as $(B,B)$-bimodules. When the bimodules have endopermutation source, $A$ and $B$ are \textit{basic Morita equivalent}. When they have trivial source, $A$ and $B$ are \textit{Puig equivalent}. Note that Puig equivalence is stronger than basic Morita, which is stronger than Morita. In particular, a basic Morita equivalence between blocks preserves the defect group and fusion system, and thus preserves the inertial quotient.

A block $B$ of $G$ with defect group $D$ is said to be \textit{inertial} if it is basic Morita equivalent to its Brauer correspondent in $\mathcal{O}N_G(D)$. If $D$ is abelian, then $B$ is nilpotent if and only if the inertial quotient of $B$ is trivial, and $B$ is called \textit{nilpotent-covered} if there is some $G \lhd H$ such that a nilpotent block $B_H$ of $H$ covers $B$. 

\begin{proposition} \label{zhou}
	Let $N \lhd G$ be finite groups and $B$ be a block of $G$ covering a block $b$ of $N$. \begin{enumerate}
		\item If $b$ is nilpotent-covered, then $b$ is inertial. \label{zhou b nilp b inert}
		\item If $p \not \divides [G:N]$ and $b$ is inertial, then $B$ is inertial.  \label{zhou b inert B inert}
	\end{enumerate}
\end{proposition} 
\begin{proof}
	(\ref{zhou b nilp b inert}) is \cite[Corollary 4.3]{Puigext2}, and (\ref{zhou b inert B inert}) is the main result of \cite{zhou}. 
\end{proof}

Given a block $b$ of $\mathcal{O}N$, the Fong-Reynolds correspondence bijects blocks of $\mathcal{O}G$ covering $b$ with blocks of $N_G(N,b)$ covering $b$, and this correspondence preserves the basic Morita equivalence class \cite[Theorem 1.18]{Sambale2014}. This lets us replace the objects we're studying with simpler objects, without losing vital information, as in the first Fong reduction:  
%, 
\begin{theorem} (\cite[6.8.3]{linckelmann}) \label{F1} 
Let $N \lhd G$ be finite groups, and let $B$ be a block of $\mathcal{O}G$ covering a block $b$ of $\mathcal{O}N$. Then there is a unique block $\tilde{B}$ of $\operatorname{Stab}_G(b)$ that is covered by $B$ and covers $b$. Further, $B$ and $\tilde{B}$ have the same defect group and fusion system.
\end{theorem}

As described in \cite[Proposition 2.2]{eaton3}, one can derive the following from \cite{Puigext}:

\begin{theorem} \label{Fong in 24} 
	Let $N \lhd G$ be finite groups, and $B$ a block of $G$ with defect group $D$ covering a $G$-stable nilpotent block of $N$ with defect group $D \cap N$. Then there is a finite group $L$ and a normal subgroup $M \lhd L$ such that: $M \cong D \cap N$; $L/M \cong G/N$; there exists $D_L \leq L$ with $D_L \cong D$ and $D_L \cap M \cong D \cap N$; and there is a central extension $\tilde{L}$ of $L$ by a $p'$-group and a block $\tilde{B}$ of $\mathcal{O}\tilde{L}$ where $\tilde{B}$ is Morita equivalent to $B$ and has defect group $\tilde{D} \cong D_L \cong D$. 
\end{theorem}

Note that the Morita equivalences in Theorems \ref{F1} and \ref{Fong in 24} are basic, by \cite[6.8.13]{linckelmann}.

\begin{corollary} (\cite[2.3]{eaton4})\label{F2 cor}
Let $N \lhd G$ be finite groups such that $N \not \leq Z(G)O_p(G)$. Let $B$ be a quasiprimitive block of $\mathcal{O}G$ covering a nilpotent block $b$ of $N$. Then there's a finite group $H$ with $[H:O_{p'}(Z(H))] < [G:O_{p'}(Z(G))]$ and a block $B_H$ of $H$ that is Morita equivalent to $B$ and has isomorphic defect group.
\end{corollary}
%\begin{proof}
%Let $b'$ be the block of $\mathcal{O}Z(G)N$ covered by $B$ and covering $b$. Then $b'$ is nilpotent, and we may assume $Z(G) \leq N$. Applying \Cref{Fong in 24}, we may take $H=\tilde{G}$, $B_H = \tilde{B}$. Then $[\tilde{G}:O_{2'}(Z(\tilde{G}))] \leq |G'|=[G:N]|D \cap N| < [G:O_{p'}(Z(G))]$.
%\end{proof}

\begin{remark} \label{reduce remark} We say that a block $B$ is \text{reduced} if it is quasiprimitive and if, whenever $B$ covers a nilpotent block of a normal subgroup $N \lhd G$, we have  $N \leq Z(G)O_p(G)$. Repeatedly applying \Cref{F1} and \Cref{Fong in 24}, lets us replace a given block with a reduced block, basic Morita equivalent to the original.
\end{remark}

The following result of K\"ulshammer describes the structure of a block when its defect group is normal in $G$:
\begin{lemma} \label{kuls} (\cite{kulscrossed})
Let $N \lhd G$ be finite groups, and let $B$ be a block of $\mathcal{O}G$ with defect group $D$ and inertial quotient $E$. If $D \lhd G$, then $B$ is Puig equivalent to a twisted group algebra $\mathcal{O}_\gamma(D \rtimes E)$, where $\gamma \in O_{p'}(H^2(E,\mathcal{O}^\times))$.
\end{lemma}

\section{Subgroup Structure} \label{section subgroup}
In this section we collect results about the group substructure of $G$. The Clifford theory in \Cref{section others} will be used to compare the block described in \Cref{main} with blocks of certain small subgroups of $G$, which are known to control the structure of $G$. 

 Recall the following definitions: a \textit{component} of $G$ is a subnormal quasisimple subgroup, and the \textit{layer} of $G$, denoted $E(G)$, is the central product of its components. The \textit{Fitting subgroup} of $G$, denoted $F(G)$, is the product of its $p$-cores for each prime $p$ dividing $|G|$, and the\textit{ generalised Fitting subgroup}, denoted $F^*(G)$, is the subgroup generated by $E(G)$ and $F(G)$. In fact, $E(G)$ and $F(G)$ commute and so $F^*(G)$ is the central product $E(G)*F(G)$. 

$F^*(G)$ is known to be self-centralising, so that $C_G(F^*(G)) \leq F^*(G)$. This provides an injective group homomorphism $G/F^*(G) \to \operatorname{Out}(F^*(G))$, and thus $F^*(G)$ controls the structure of $G$ \cite[\S 11]{aschbacher}. Further, we will show that there is a unique component $E(G)=L \lhd G$, and we will be able to assume $O_2(G)=1$, $O_{2'}(G)=F(G)=Z(G)$, so that $G/F^*(G) \leq \operatorname{Out}(L)$. This will allow us to prove \Cref{main} using the classification of blocks of quasisimple groups with abelian defect groups from \cite{ekks}. 

$G$ acts by conjugation on the set of components, permuting them, and $B$ covers a block $b$ of each $G$-orbit of components. Further, $b$ covers a block of each component in that orbit, and these blocks are pairwise isomorphic with isomorphic defect groups. Note also that distinct components intersect only in their centers.

\begin{theorem} \label{tool central product D1D2=D nilp etc |E|=prod}
	(\cite[7.5]{Sambale2014}) Let $G = H_1 * H_2$ be a central product of finite groups. Let $B$ be a block of $G$, and let $b_i$ be a $B$-covered block of $H_i$ with defect group $D_i$. Then  $D_1D_2$ is a defect group of $B$, and $B$ is nilpotent if and only if both $b_1$ and $b_2$ are. If $H_1 \cap H_2$ is a $p'$-group, then $B \cong b_1 \otimes b_2$. \end{theorem}

\begin{corollary} \label{tool no nilp comps}
	Let $B$ be a reduced block of a finite group $G$, with abelian defect group $D$. Let $E(G)=L_1 * ... *L_t$ be the layer of $G$, so each $L_i$ is quasisimple, and let $b_E$ be the $B$-covered block of $E(G)$, and $b_i$ be a $b_E$-covered block of $L_i$, for each $1 \leq i \leq t$. Then $b_i$ is not nilpotent. 
\end{corollary}
\begin{proof}
	Suppose, for a contradiction, that such a block $b_i$ is nilpotent. Then all the components in the $G$-orbit of $b_i$ are also nilpotent. Thus, $B$ covers a nilpotent block of the normal subgroup $N$ generated by this orbit of components, by \Cref{tool central product D1D2=D nilp etc |E|=prod}. Since $B$ is reduced, this implies $L_i \leq N \leq O_p(G)Z(G)$ - a contradiction. 
\end{proof}

\begin{lemma} \label{tool Z(G) < Z(F^*(G)) < O2(G)Z(G)}
	Let $G$ be a finite group and $B$ a $p$-block of $\mathcal{O}G$. If $B$ is reduced, then $O_{p'}(G) \leq Z(G) \leq Z(F^*(G)) \leq O_p(G)Z(G)$. 
\end{lemma}
\begin{proof}
Since $C_G(F^*(G)) \leq F^*(G)$, we have $Z(G) \leq C_G(F^*(G)) \leq F^*(G)$, implying $Z(G) \leq Z(F^*(G))$. Further, $B$ covers a block $b_Z$ of $Z(F^*(G))$. Since $Z(F^*(G))$ is abelian, $b_Z$ is nilpotent, and since $B$ is reduced, this implies $Z(F^*(G)) \leq Z(G)O_p(G)$. Similarly, $B$ covers a block $b_{p'}$ of $O_{p'}(G)$. By \Cref{alp15} (\ref{alpintersect}), $b_{p'}$ has defect group $D \cap O_{p'}(G)$, which must be trivial since $D$ is a $p$-group. Thus, $b_{p'}$ is nilpotent and so $O_{p'}(G) \leq Z(G)O_{2}(G)$, implying $O_{p'}(G) \leq Z(G)$. \end{proof}

\begin{proposition} \label{O2}
Let $B$ be a reduced block of a finite group $G$, with abelian defect group $D$ and inertial quotient $E$. Suppose $D$ is not normal in $G$, and let $H< D$ such that $H \lhd G$. Then $p \not \divides [G:C_G(H)]$, and $H$ is centralised by a non-trivial normal subgroup of $E$. 
\end{proposition}
\begin{proof}
We adapt the argument from \cite[Proposition 6.4]{wuzzz}: suppose that $H$ is non-trivial. Let $C$ denote $C_G(H)$, and $b_C$ be a block of $C$ covered by $B$. Since $D$ is abelian and $H \leq D$, we have $D \leq C_G(D) \leq C$, so $b_C$ has defect group $D$. Further, \Cref{alp15} (\ref{alp CG(D)<N}) tells us that ${b_C}^G=B$ and $B$ is the unique block of $G$ covering $b_C$. Since $C_G(D) \leq C$, we may adjust the choice of $b_C$ so that $(D,B_D)$ is a maximal Brauer pair of $b_C$. 

$B$ is the unique block of $G$ covering $b_C$, so \Cref{alp15} (\ref{alpblockabove}) tells us $DC/C$ is a Sylow-$p$-subgroup of $G/C$. However, $D \leq C$, so this implies $[G:C]$ is not divisible by $p$. Now, let $E_C$ denote the inertial quotient of $b_C$. Since $N_{C}(D,B_D):=N_G(D,B_D) \cap C$, we have $N_{C}(D,B_D) \lhd N_G(D,B_D)$. Further, $C_G(D)\leq C$ and so $C_G(D)=C_{C}(D)$. This implies $$E_C :=\frac{N_ {C}(D,B_D)}{C_G(D)} \text{ \LARGE $\lhd$ \normalsize} \frac{N_G(D,B_D)}{C_G(D)}=:E.$$  If $E_C=1$, then $b_C$ is nilpotent. Since $B$ is reduced, this implies $D \leq C \leq O_p(G)Z(G)$, contradicting our assumption that $D$ is not normal in $G$.  Of course, $H$ is centralised by $C$, so there is a non-trivial normal subgroup $E_C$ of $E$ that centralises $H$, as claimed.
\end{proof}

\Cref{O2} will be useful when considering the generalised Fitting subgroup, since $O_p(G)$ is a normal subgroup of $G$ that is known to lie in the defect group.

\begin{lemma} \label{tool Mashke} \label{tool C_E(Q)=E or G=O^2(G) if C_D(E)=1} 
	Let $N \lhd G$ be finite groups such that $[G:N]$ is a power of a prime $p$. Let $b$ be a $G$-stable $p$-block of $N$, and let $B$ be a block of $G$ covering $b$. Suppose $B$ has inertial quotient $E$ and abelian defect group $D$ such that $D \cap N$ has a complement $D'$ in $D$. Then $D'$ is centralised by $E$. In particular, if $E$ acts freely on the non-trivial elements of $D$, then $O^p(G)=G$. 
\end{lemma}
\begin{proof}
	 We follow the proof of the main theorem of \cite{koshitanikulshammer}. Since $[G:N]$ is a power of $p$,  \Cref{index p} tells us $B$ is the unique block of $G$ covering $b$, and by \cite[V]{feit} we can write $G=DN$. Since $N \lhd G$, we know $D \cap N$ is an $N_G(D,B_D)$-stable subgroup of $D$. Since $D$ is abelian, we also know that $D \leq N_G(D,B_D)$.
	 
Notice that $D$ is an abelian $p$-group, and $E$ a $p'$-group of automorphisms of $D$. Thus we can apply \cite[3.3.2]{gorenstein}, a corollary of Gorenstein's proof of Maschke's theorem, which tells us that if an $E$-invariant direct factor of $D$ has a complement, then that complement is $E$-invariant. Since $D \cap N$ is $E$-invariant, this tells us $D'$ is $E$-invariant. Thus $D' \lhd N_G(D,B_D)$. 
	
Since $G=DN$, and so $G/N$ is abelian, we have that $[G,G] \leq N$. However, $D'$ is $N_G(D,B_D)$-stable, so $$[D',N_G(D,B_D)] \leq D' \cap [G,G] \leq D' \cap N=1.$$ That is, $D'$ commutes with $N_G(D,B_D)$, and so $C_E(D')=E$. In particular, if $C_D(E)=1$ then $G$ has no normal subgroups of $p$-power index, i.e. $O^p(G)=G$. \end{proof}

The following describes the behaviour of components with respect to the defect group:

\begin{lemma} \label{tool LicapLj and DcapZLi in O2ZG if O2 is 1 then int odd n D is directprod}
	Let $G$ be a finite group and $B$ be a reduced block of $G$ with abelian defect group $D$ and inertial quotient $E$. Let $L_1$,...,$L_t$ be the components of $G$. Then the following hold: \begin{enumerate} 
		 \item $L_i \cap L_j \leq Z(E(G)) \leq O_p(G)Z(G)$, for every $i \neq j$. \label{lem1}
		\item $D \cap Z(L_i) \leq D \cap Z(E(G)) \leq O_p(G)$, for all $i$. \label{lem2}
		\item If $L_n$ and $L_m$ are components of $G$ that are not $G$-conjugate, write $N:=\langle \{{}^g L_n \mid g \in G \} \rangle$ and $M:=\langle \{ {}^g L_m \mid g \in G \} \rangle$. Then $D \cap (N \cap M) \leq O_p(G)$. \label{lem3} 
		\item If $O_p(G)=1$, then $L_i \cap L_j$ is a $p'$-group and $D \cap E(G)=(D \cap L_1) \times ... \times (D \cap L_t)$. \label{lem4}
	\end{enumerate}\end{lemma}
\begin{proof}
Since $E(G)$ is a central product,  $L_i \cap L_j \leq Z(L_i)$, and $Z(L_i) \leq Z(E(G))$. Since $Z(E(G))$ is abelian, $B$ covers a nilpotent block of $Z(E(G))$, and since $B$ is reduced this implies $Z(E(G)) \leq O_p(G)Z(G)$. Thus $L_i \cap L_j \leq O_p(G)Z(G)$ and $D \cap Z(L_i) \leq O_p(G)Z(G)$. However, $D \cap Z(G) \leq O_p(G)$, so $D \cap Z(L_i) \leq O_p(G)$. This proves statements (\ref{lem1}) and (\ref{lem2}).
	
This implies that $N \cap M$, which is the intersection of two conjugacy classes of components, lies in $Z(E(G))$. Therefore $D \cap (N \cap M) \leq D \cap Z(E(G)) \leq O_p(G)$, proving (\ref{lem3}). 
	
	Finally, if $O_p(G)=1$ then $Z(G)$ is a $p'$-order group by \Cref{tool Z(G) < Z(F^*(G)) < O2(G)Z(G)}, implying that $L_i \cap L_j$ is a $p'$-group by (\ref{lem1}). Further, $D$ is a $p$-group, so if $O_p(G)=1$ then $D \cap Z(L_i)=D \cap (N \cap M)=1$ by (\ref{lem2}) and (\ref{lem3}). Since distinct components intersect only in their centres, this implies $D \cap E(G)=(D \cap L_1) \times ... \times (D \cap L_t)$.	\end{proof}

\subsection{Blocks of Quasisimple Groups With Abelian Defect Group}
For the results above to be useful, we need to relate the subgroup structure of $G$ to the block substructure. To that end, we present the classification of $2$-blocks of quasisimple groups with abelian defect group from \cite{ekks}, followed by some clarifications we can make in our case.

\begin{theorem} \label{EKKS structure}
	(\cite[6.1]{ekks}) Let $L$ be a quasisimple group, and $b$ be a $2$-block of $\mathcal{O}L$ with abelian defect group $D$. Then one of the following holds. \begin{enumerate}			
		\item \label{EKKSi} $L/Z(L)$ is one of ${}^2G_2(q)$, $J_1$, or $SL_2(2^a)$, and $b$ is the principal block of $L$. 
		
		\item \label{EKKSii} $L$ is $Co_3$, $D \cong (C_2)^3$, and $b$ is the unique non-principal block of $L$.
		
		\item \label{EKKSiii} $b$ is nilpotent-covered, i.e. there exists a finite $\tilde{L} \rhd L$ with $Z(\tilde{L}) \geq Z(L)$ such that a nilpotent block of $\tilde{L}$ covers $b$.
		
		\item \label{EKKSiv} For some $M=M_0 \times M_1 \leq L$, $b$ is Morita equivalent to a block $b_M$ of $\mathcal{O}M$ such that the defect groups of $b_M$ are isomorphic to $D$. Further, $M_0$ is abelian, and the block of $\mathcal{O}M_1$ covered by $b_M$ has Klein $4$-defect groups. In particular, $b$ is Morita equivalent to a tensor product of a nilpotent block and a block with Klein $4$-defect group.	\end{enumerate}	\end{theorem}

%Note that these cases describe different aspects of the structure: cases (\ref{EKKSi}) and (\ref{EKKSii}) describe the quasisimple group up to isomorphism, whereas cases (\ref{EKKSiii}) and (\ref{EKKSiv}) describe the block structure, and they describe it in relation to other blocks. In fact, we will narrow down the group structure in these latter two cases, and in \Cref{allMor} we'll see that some of the blocks in the former two cases are Morita equivalent. We ask the reader to bear in mind that at times we will be considering blocks up to Morita equivalence, and at other times we will be considering them up to isomorphism of the underlying group. 

Since $D$ is homocyclic, and since we will be able to assume that the block of the quasisimple group is not nilpotent, we can make the following clarifications to \Cref{EKKS structure}:

\begin{corollary} \label{ekkssimple}
Let $L$ be a quasisimple group, and $b$ be a non-nilpotent $2$-block of $\mathcal{O}L$ with defect group $D \cong (C_{2^m})^n$, where $n,m \geq 1$. Then the following table describes the six possible situations that can occur, including the isomorphism type of $L$, $D$, $\operatorname{Out}(L)$, and the inertial quotient $\mathbb{E}_b$ of $b$, where appropriate:

\addtolength{\parskip}{-0.6cm}\begin{table}[H] \small
	\centering % used for centering table
	\begin{tabular}{c c c c c l} % centered columns (4 columns)
		\hline\hline  %inserts double horizontal lines
		\textbf{Group} 				& \textbf{Block}& $D$		& \footnotesize{Out($L)$}& $\mathbb{E}_b$ 	& Notes\\ [1ex]
		\hline 
		$SL_2(2^n)$ 	& Principal 	& $(C_2)^n$	& $C_n$		&	$C_{2^n-1}$	&	 $n \geq 2$	\\ [1ex]
		${}^2G_2(q)$& Principal 	& $(C_2)^3$	& $C_q$		&	$C_7 \rtimes C_3$	& $q=3^r$, $r \geq 3$. 	\\ [1ex]
		$ J_1$ 		& Principal  	& $(C_2)^3$	& 	$1$	&	$C_7 \rtimes C_3$	&		\\ [1ex]
		$ Co_3$ 			& \small{Non-principal}	& $(C_2)^3$	& 	$1$	& 	$C_7 \rtimes C_3$	& \footnotesize{Unique non-principal}\\[1ex]
 \footnotesize{$L/Z(L)$ type $A_t$ or $E_6$}		& \footnotesize{Nilpotent-covered}& 	-	& 	-	&	-	&	\\ [1ex]
	$M_0 \times M_1 \leq L$ & $b_0 \otimes b_1$	& 	$D_0 \times C_2^2$	&	-	& $C_3$ & \dagger \\ 		\hline 	\end{tabular}\end{table} \noindent \dagger { } In the final case,  $b$ is Morita equivalent to $b_0 \otimes b_1$, where $b_i$ is a block of $M_i$. Further, $M_0$ is abelian, $b_1$ is Puig equivalent to $\mathcal{O}A_4$ or $B_0(\mathcal{O}A_5)$, and $\mathbb{E}_b \cong C_3$ centralises $D_0=D \cap M_0$. If $m=1$, then $L$ is of type $D_t(q)$ or $E_7(q)$ where $q$ is an odd prime power and $t/2$ is odd. \addtolength{\parskip}{+0.6Cm}
\end{corollary}
\begin{proof}
We apply \Cref{EKKS structure}. In case (\ref{EKKSi}) of \Cref{EKKS structure}, the groups ${}^2G_2(q)$, $J_1$, and $SL_2(2^n)$ for $n > 2$, have trivial Schur multiplier. Further, if $L/Z(L) \cong SL_2(4)$ then $L=L/Z(L)$ since the double cover of $SL_2(4)$ has non-abelian defect group. Thus $L/Z(L)=L$ in case (\ref{EKKSi}). The principal block of ${}^2G_2(q)$ has defect group $(C_2)^3$ by \cite{Reee,2G2KKL}, and the character table of $J_1$ in \cite{janko} demonstrates that the $2$-blocks of $J_1$ have defect $0$, except for the principal $2$-block, whose defect group is $(C_2)^3$ \cite{J1LandM}.

The Conway group $Co_3$ also has trivial Schur multiplier, and has a unique non-prinipal $2$-block with defect group $(C_2)^3$, by \cite[1.5]{KoshitaniaBrouS}.

If $b$ is nilpotent-covered but not nilpotent, then the hypothesis of \cite[5.4]{ekks} applies to \cite[4.2]{ekks}, so in case (\ref{EKKSiii}), $L/Z(L)$ is of type $A_t(q)$ or $E_6(q)$, where $q$ is a power of an odd prime.

In case (\ref{EKKSiv}) of \Cref{EKKS structure}, there exists $M_0 \times M_1 \leq L$ with $M_0$ abelian, and $b$ is Morita equivalent to $b_0 \otimes b_1$, where $b_i$ is a block of $M_i$ and $b_1$ has Klein four defect groups. The main result of \cite{kleinfour} tells us that $b_1$ is Puig equivalent to the principal blcok of $\mathcal{O}A_5$, $\mathcal{O}A_4$, or $\mathcal{O}(C_2)^2$. However, we can exclude the latter since $b_0$ is nilpotent but $b$ is not. The Morita equivalence between $b$ and $b_0 \otimes b_1$ is given by the Bonnaf\'e-Dat-Rouquier correspondence, which preserves the defect group and inertial quotient up to isomorphism \cite{bonnaferouquierdat}, implying that $b$ has inertial quotient $C_3$ which centralises $D \cap M_0$, and which transitively permutes the non-trivial elements of $D \cap M_1 \cong (C_2)^2$. Finally, if $D$ is elementary abelian,  then \cite[Proposition 5.3]{ekks} tells us $L$ is of type $D_t(q)$ or $E_7(q)$ for $q$ an odd prime power and $t/2$ odd. \end{proof}

\begin{remark} \label{allMor}
It is proven in \cite[Theorem 1.5, Lemma 4.2 (xi)]{KoshitaniaBrouS} that the non-principal block of $Co_3$ with defect group $(C_2)^3$ is Puig equivalent to the principal block of $\operatorname{Aut}(SL_2(8))$. 

Further, according to \cite[Example 3.3, Remark 3.4]{okuyamasomeexamples}, it had been essentially proved in \cite{Reee} that the principal block of $\operatorname{Aut}(SL_2(8))$ and the principal block of ${}^2G_2(q)$ are Puig equivalent (see \cite[Theorem 1.6]{KoshitaniaBrouS}).
\end{remark}

%In particular, note that, as $q$ varies, the principal blocks of ${}^2G_2(q)$ share a single Morita equivalence class, and that ${}^2G_2(3) \cong \operatorname{Aut}(SL_2(2^n))$. 

%\begin{remark} 
%It's conjectured that all Morita equivalences can be achieved by a bimodule with endopermutation source, i.e. that a Morita equivalence between blocks implies the existence of a basic Morita equivalence between those blocks. In \Cref{main}, our classification is up to basic Morita equivalence, except possibly when $m=1$ and $n=|F|=3$. This is because the non-principal block of $Co_3$ with defect group $(C_2)^3$, and the principal blocks of $\operatorname{Aut}(SL_2(8))$ and ${}^2G_2(q)$, are pairwise Morita equivalent (see \cite[Theorem 1.5]{KoshitaniaBrouS} and \cite[3.3]{okuyamasomeexamples}), but are not known to be basic Morita equivalent. 
%\end{remark}

\subsection{Inertial Quotient Analysis}
We saw in \Cref{alp15} that when a block $B$ covers a block $b$, the defect group of $b$ is easily understood in terms of the defect group above. The inertial quotient is not so well behaved, but in this subsection we show that there are circumstances where we can usefully compare the inertial quotients of $B$ and $b$.

\begin{lemma} \label{outsl}
Let $L  \cong SL_2(2^n)$ and $L \leq G \leq \operatorname{Aut}(L)$. Let $B$ be a block of $G$ covering the principal block $b$ of $L$. Then $G=L \rtimes F$ where $F \leq C_n$, $b$ has inertial quotient $E_b \cong C_{2^n-1}$, and $B$ has inertial quotient $\mathbb{E}=E_b \rtimes F$. 
\end{lemma}
\begin{proof}
The center of $SL_2(2^n)$ is trivial, so $G \leq \operatorname{Aut}(L) \cong SL_2(2^n) \rtimes \operatorname{Out}(SL_2(2^n))$, where $\operatorname{Out}(SL_2(2^n))\cong C_n$ are the field automorphisms of the underlying field \cite[Table 5]{atlas}. Since $b$ is principal, the defect groups of $b$ are precisely the Sylow-$2$-subgroups of $SL_2(2^n)$. Therefore, consider the defect group of the form  $$D = \biggl\{ \begin{pmatrix} 1 & f \\ 0 & 1 \end{pmatrix} \text{ } | \text{ } f \in \mathbb{F}_{2^n} \biggr\} \cong (C_2)^n .$$ By definition, the field automorphisms will normalise but not centralise $D$. That is, $C_L(D)=C_G(D)$ and $[N_G(D):N_L(D)]=[G:L]$, implying that $E_b \lhd \mathbb{E}$ with $[\mathbb{E}:E_b]=[G:L]$. The inertial quotient of $b$ is generated by a Singer cycle, so $E_b \cong C_{2^n-1}$, and, since $\mathbb{E}$ has odd order, \Cref{homo Sing lem} (\ref{'1}) then tells us $C_{2^n-1} \leq \mathbb{E} \leq C_{2^n-1} \rtimes C_n$. 
\end{proof}

The following is an extension of \cite[Lemma 6.3]{wuzzz}, whose sophisticated proof applies more broadly than is stated in that paper. Following \cite{koshitanikulshammer}, denote the stabiliser in $C_X(D)$ of a block $b$ by $C_X(D)_{b}$.

\begin{lemma} \label{tool IQ<IQ conditions} \label{IQlem} 
	Let $N \lhd G$ be finite groups. Let $B$ be a block of $G$ with inertial quotient $E_B$ covering a $G$-stable block $b$ of $N$ with inertial quotient $E_b$, such that $B$ and $b$ share an abelian defect group $D$. Let $b_D$ be a Brauer correspondent of $b$ in $C_N(D)$. Then the following statements are true: \begin{enumerate} \item \label{IQ tool case1}If $C_G(D)=C_N(D)$, then there is a monomorphism $E_b \hookrightarrow E_B$ whose image is normal. 
		\item \label{IQ tool case2} If $C_N(D) \neq C_G(D)$ and $C_G(D)_{b_D}=C_N(D)$, then there is an monomorphism $E_b \hookrightarrow E_B$ whose image is normal.
		\item \label{IQ tool case3} If $[C_G(D):C_N(D)]=[G:N]$ and $C_G(D)_{b_D}=C_G(D)$, then there is an monomorphism $E_B \hookrightarrow E_b$. 
		
		%		\item \color{red} If $C_N(D) \lhd C_G(D)$ with $[C_G(D):C_N(D)] < [G:N]$, and $C_G(D)_{b_D}=C_G(D)$, then $E_B \hookrightarrow E'$ where $E'$ is the iq of a block of $NC_G(D)$? But is $NC_G(D) \lhd G$?
		
		\color{black}
	\end{enumerate} In particular, if $[G:N]$ is prime, then either there is a normal subgroup of $E_B$ that is isomorphic to $E_b$, or there is a subgroup of $E_B$ that is isomorphic to $E_b$.	\end{lemma}

%OR IF $N_G(D,b_D)=N_N(D,b_D)C_G(D)$

\begin{proof}
	We choose Brauer correspondents $B_D$ and $b_D$ for $B$ and $b$ in $C_G(D)$ and $C_N(D)$, respectively, such that $B_D$ covers $b_D$ (we are able to do so by \cite[Lemma 2.1]{kessarmalle}, as explained in \cite[Lemma 2.5]{CesC2^5}). 
	
	Suppose $C_G(D)=C_N(D)$. This implies $B_D=b_D$ and so $E_b \lhd E_B$.
	
	Suppose $C_N(D) \neq C_G(D)$ and $C_G(D)_{b_D}=C_N(D)$. Let $I$ be a complete set of coset representatives of $C_N(D)$ in $C_G(D)$. Since $B_D$ covers $b_D$, we have $B_D=\sum_{g \in I}({}^gb_D)$, and $({}^hb_D)b_D=0$ for any $h \in I \backslash \{C_N(D)\}$. Therefore, if $g \in N_G(D,b_D)$, then $g$ must stabilise $B_D$, so $N_G(D,b_D) \subset N_G(D,B_D)$. For each $y \in N_G(D,B_D)$, there exists some $z \in C_G(D)$ such that $z^{-1}y$ centralises $D$. Thus $N_G(D,B_D) \subset N_G(D,b_D)C_G(D)$ and so $N_G(D,B_D)=N_G(D,b_D)C_G(D)$. Therefore, the inclusion $N_N(D,b_D) \subset N_G(D,B_D)$ induces an injective group homomorphism $E_b \hookrightarrow E_B$, whose image is normal. 
	
	Suppose $[C_G(D):C_N(D)]=[G:N]$ and $C_G(D)_{b_D}=C_G(D)$ (i.e. $b_D$ is $C_G(D)$-stable). Then $b_D$ is a central idempotent in $\mathcal{O}C_G(D)$, and is the unique block of $C_N(D)$ covered by $B_D$. Thus, $N_G(D,B_D) \subset N_G(D,b_D)$. Since $[C_G(D):C_N(D)]=[G:N]$, we can write $G=NC_G(D)$, and thus $N_G(D,b_D)=N_N(D,b_D)C_G(D)$. Then the inclusion $N_G(D,B_D) \subset N_G(D,b_D)$ induces an injective group homomorphism $ {N_G(D,B_D)}/{C_G(D)} \hookrightarrow {N_G(D,b_D)}/{C_G(D)} \cong {N_N(D,b_D)}/{C_N(D)}.$
	
	Finally, if $[G:N]$ is prime, then either $C_G(D)=C_N(D)$, or $[C_G(D):C_N(D)]=[G:N]$. In the former case, (\ref{IQ tool case1}) implies there is a monomorphism $E_b \hookrightarrow E_B$ whose image is normal. In the latter, either $C_G(D)_{b_D}=C_N(D)$, in which case (\ref{IQ tool case2}) implies $E_b \hookrightarrow E_B$ with normal image, or $b_D$ is $C_G(D)$-stable and (\ref{IQ tool case3}) tells us there is a monomorphism $E_B \hookrightarrow E_b$. \qedhere  \end{proof}

When $G/N$ is abelian, \Cref{tool IQ<IQ conditions} allows us to partially relate the inertial quotient of a reduced block $B$ of $G$ to the inertial quotient of a $B$-covered block of $N$: 

\begin{theorem} \label{iq tool EK}\label{IQcor}
	Let $N \lhd G$ be finite groups such that $G/N$ is abelian, and let $B$ be a reduced block of $G$ with inertial quotient $E_B$ covering a block $b$ of $N$ with inertial quotient $E_b$, such that $B$ and $b$ share an abelian defect group $D$. Then there is a chain $N \lhd K \lhd G$ and a block $b_K$ of $K$ with inertial quotient $E_{b_K}$ such that there are monomorphisms $E_b \hookleftarrow E_{b_K} \hookrightarrow E_B$, where the image of $E_{b_K}$ is normal in $E_B$. 
\end{theorem}	
\begin{proof}
Since $G/N$ is abelian, we have that for all $C \leq G$, $CN/N \lhd G/N$  and thus $NC \lhd G$. 

If $C_N(D)=C_G(D)$ then \Cref{IQlem} (\ref{IQ tool case1}) tells us there is a monomorphism $E_b \hookrightarrow E_B$ whose image is normal, in which case we are done by setting $K$ as either $N$ or $G$. 
	
If $C_N(D) \neq C_G(D)$, then define $L:=NC_G(D)_{b_D}$. Thus we have $N \lhd L$ and $[C_L(D):C_N(D)]=[L:N]$ with $C_L(D)_{b_D}=C_L(D)$. Thus, \Cref{IQlem} (\ref{IQ tool case3}) tells us there is a monomorphism $E_b \hookleftarrow E_{b_L}$. 
	
	Next, consider the chain $L=K_0 \lhd K_1 \lhd K_2 \lhd ... $ defined by $K_{i+1}:=K_{i}C_G(D)_{{b_{i}}_D}$, where $B$ covers a block $b_i$ of each $K_i$, and ${b_i}_D$ is a Brauer correspondent of $b_i$ in $C_{K_i}(D)$. Since $[G:K]$ is finite, there must be some $j \geq 0$ such that $K_j = K_{j+1}$, in which case $C_G(D)_{{b_j}_D} \leq K_j$, and $[C_{K_j}(D):C_{K_0}(D)]=[K_j:K_0]$ with $C_{K_j}(D)_{{b_j}_D}=C_{K_j}(D)$, and so \Cref{IQlem} (\ref{IQ tool case3}) tells us that there is a monomorphism $E_{b_{L}}=E_{b_0} \hookleftarrow E_{b_j}$. 
	
	Now, consider $M=K_jC_G(D)$. Since $C_G(D)_{{b_j}_D} \leq K_j$, we have that $[C_M(D):C_{K_j}(D)]=[M:K_j]$ and $C_M(D)_{{b_j}_D}=C_{K_j}(D)$, and so \Cref{IQlem} (\ref{IQ tool case2}) tells us there is a monomorphism $E_{b_{j}} \hookrightarrow E_{b_M}$ whose image is normal.
	
	Finally, since $C_G(D) \leq M$, we have that $C_G(D)=C_M(D)$, and so \Cref{IQlem} (\ref{IQ tool case1}) tells us there is a monomorphism $E_{b_M} \hookrightarrow E_B$ whose image is normal.
	
	In summary, we have a chain $N \lhd K_0 \lhd ... \lhd K_j \lhd M \lhd G$, with a corresponding chain of inertial quotients of $B$-covered blocks and associated monomorphisms $$E_{b} \hookleftarrow E_{b_0} \hookleftarrow ... \hookleftarrow E_{b_j} \hookrightarrow E_{b_M} \hookrightarrow E_B,$$ where the images of the final two monomorphisms are normal in $E_{b_M}$ and in $E_B$, respectively. In particular, $E_{b_j}$ is isomorphic to a subgroup of $E_{b}$ and to a normal subgroup of $E_B$.
\end{proof}

\begin{corollary}
	Assume the conditions of \Cref{IQcor}. Suppose that $B$ does not cover any nilpotent block of any normal subgroup containing $N$. If $|E_b|$ is prime, then there is a monomorphism $E_b \hookrightarrow E_B$ whose image is normal, and if $|E_B|$ is prime then there is a monomorphism $E_B \hookrightarrow E_b$. If $|E_b|$ and $|E_B|$ are both prime, then $E_b \cong E_B$.
\end{corollary}

\section{Proof of the main theorem} \label{section proof}
We will conclude this section by proving \Cref{main}. First, we apply the results we've gathered to the group and block described in \Cref{main}. We start with a robust result regarding the intersection of the defect group with normal subgroups, a consequence of the unique action of a Singer cycle. We'll use this to show that $G$ contains no normal subgroups of order or index a positive power of $2$, and that $G$ contains a unique component. 

\begin{lemma} \label{DcapN}
Assume the conditions of \Cref{main}. Let $N \lhd G$. Then $D \cap N$ is trivial or $D \cap N \cong (C_{2^{m'}})^n$ where $m' \leq m$. In particular, if $D \cap N$ is non-trivial then $\Omega(D) \leq N$.
\end{lemma}
\begin{proof}
Suppose $D \cap N$ is non-trivial. Then it contains an element $d$ of order $2$, i.e. a non-trivial element of $\Omega(D)$. \Cref{Sings} tells us that $\mathbb{E}$ acts transitively on the non-trivial elements of $\Omega(D)$, so the $\mathbb{E}$-orbit of $d$ contains $\Omega(D)$. Since $\mathbb{E}$ acts by conjugation and $N$ is normal, this implies the orbit is contained in $N$, so $\Omega(D) \leq N$. Therefore we can write  $D \cap N \cong C_{2^{m_1}} \times C_{2^{m_2}} \times ... \times C_{2^{m_n}}$ where $1 \leq m_i \leq m$ for each $1 \leq i \leq n$. We can choose a set of generators for $D$ that lie in a single orbit under conjugation by $\mathbb{E}$, implying that $m_i=m_j$ for all $1 \leq i,j \leq n$. 
\end{proof}

\begin{lemma} \label{O22}
Assume the conditions of \Cref{main}. If $B$ is reduced and $D$ is not normal in $G$, then $O_2(G) =1$. 
\end{lemma}
\begin{proof}
\Cref{O2} tells us that $O_2(G)$ is centralised by a non-trivial normal subgroup $H$ of $\mathbb{E}$, but \Cref{homo Sing lem} tells us that, under these conditions, $C_D(H) = 1$. Therefore $O_2(G) =1$. 
\end{proof}

We will often be able to make the assumption that $D$ is not normal in $G$, as we do in \Cref{O22}, because when $D$ is normal, we fully understand the structure of $B$:

\begin{proposition} \label{normal}
	Assume the conditions of \Cref{main}. If $D$ is normal in $G$, then $B$ is Puig equivalent to $\mathcal{O}(D \rtimes \mathbb{E})$. 
\end{proposition}
\begin{proof}
	When $D$ is normal in $G$, \Cref{kuls} tells us that $B$ is Puig equivalent to a twisted group algebra $\mathcal{O}_\gamma(D \rtimes \mathbb{E})$, where $\gamma \in O_{p'}(H^2(\mathbb{E},\mathcal{O}^\times))$. Note that if $\mathbb{E}$ has trivial Schur multiplier, then $\gamma$ must be trivial. 
	
	A finite group is said to have deficiency zero if it admits a presentation with an equal number of generators and relations, and \cite[p.87]{johnson1997} tells us that finite groups of deficiency zero have trivial Schur multiplier. Considering the presentation  $E \rtimes F=\langle x, y \divides yxy^{-1}=x^2, y^{|F|}=1 \rangle $, $\mathbb{E}$ has deficiency zero, implying $\gamma$ is trivial and so $B$ is Puig equivalent to  $\mathcal{O}(D \rtimes E)$. 
\end{proof}

Now we focus our attention on the components of $G$.

\begin{lemma} \label{E acts trans}
Assume the conditions of \Cref{main}. If $B$ is reduced and $D$ is not normal in $G$, then $E$ acts transitively on the components. Further, the action is faithful if there is more than one component. %, and $C_D(\mathbb{E}) \cap E(G)$ is contained in the intersection of all the components of $G$. 
\end{lemma}
\begin{proof}
Suppose that $G$ contains components $L_n$ and $L_m$ that are not conjugate in $G$, and denote by $N$ and $M$ the subgroups generated by the $G$-conjugates of $L_n$ and $L_m$, respectively. If $D \cap N=1$ then $B$ covers a nilpotent block of $N$, which covers a nilpotent block of $L_n$, contradicting \Cref{tool no nilp comps}. Thus, by \Cref{DcapN}, $\Omega(D) \leq N$. By the same argument, $\Omega(D) \leq M$, implying $\Omega(D) \leq N \cap M$. However, \Cref{tool LicapLj and DcapZLi in O2ZG if O2 is 1 then int odd n D is directprod} tells us that $D \cap (N \cap M) \leq O_2(G)$, and \Cref{O22} tells us $O_2(G)=1$, implying $\Omega(D) =1$. This is a contradiction. Thus the components of $G$ are pairwise conjugate in $G$. 

This implies $\Omega(D) \cap L_i \cong \Omega(D) \cap L_j$ for each pair of components $L_i$ and $L_j$. Since $E$ acts transitively and faithfully on the non-trivial elements of $\Omega(D)$ by \Cref{Sings} (\ref{Singsregular}), this implies $E$ acts transitively and faithfully on the set of components. \qedhere

% Then $N \cap M \lhd G$, so \Cref{DcapN} tells us that either $\Omega(D) \cap (N \cap M)=1$ or $\Omega(D) \leq (N \cap M)$. However, \Cref{tool LicapLj and DcapZLi in O2ZG if O2 is 1 then int odd n D is directprod} tells us $D \cap (N \cap M) \leq O_2(G)$, while \Cref{O22} tells us $O_2(G) =1$. Thus, $D \cap (N \cap M)=1$ and so either $D \cap N=1$ or $D \cap M=1$. Since $B$ covers a block each of $N$ and $M$, this implies one of these blocks has either trivial defect group, and thus is nilpotent, contradicting \Cref{tool no nilp comps}. 

% Now, let $1 \neq d \in C_D(E)$ and suppose $d \in L_d$. Let $L_c$ be a distinct component. There is some $e \in E$ such that ${}^eL_d = L_c$, but $e$ centralises $d$. Therefore $d \in L_d \cap L_c$. This applies to every component, implying that $d \in \bigcap L_i$. 
\end{proof}

We will use \Cref{E acts trans} to prove that $G$ contains a unique component, but to do so we need the following:

\begin{lemma} \label{tool E in wreath nope}
Let $n=tk$ be an integer with $t,k \geq 2$. Then there is no embedding of $C_{2^n-1}$ into $C_{2^k-1} \wr C_t$. 
\end{lemma}
\begin{proof}
Recall that the exponent of a group is the least common multiple of the orders of each of its elements. The exponent of $C_{2^k-1} \wr C_t$ is at most $(2^k-1)  t$.  Thus $$(2^{kt}-1) - t(2^k-1)=2^{kt}-t2^k+(t-1) > 2^{kt}-t2^k$$ $$ \geq 2^{k+t}-t2^k=2^k(2^t-t) > 0$$ 	Thus $2^{kt}-1$ is larger than the exponent of $C_{2^k-1} \wr C_t$, so the wreath product can't contain an element of this order.
\end{proof}

Recall that a group $H$ is \textit{indecomposable} if it it has no proper non-trivial direct factor. Note that quasisimple groups are indecomposable. 

\begin{lemma} \label{1comp}
	Assume the conditions of \Cref{main}. If $B$ is reduced then $G$ contains a unique, normal quasisimple group $L$. 
\end{lemma}
\begin{proof}
Write $E(G)=L_1 * ... * L_t$, so that $t$ is the number of components of $G$. Suppose $t>1$. By \Cref{E acts trans}, $E$ acts transitively and faithfully on the set of components, and \cite[Lemma 1.7]{harrisEG} tells us that this action can be lifted to an action on $L_1 \times ... \times L_t$. That is, there is an embedding $E \hookrightarrow \operatorname{Aut}(L_1 \times ... \times L_t)$. Since the action is transitive, $L_i \cong L_j$ for each $1 \leq i,j \leq t$. Since quasisimple groups are indecomposable, \cite[Corollary 3.3]{automorphisms} tells us that, if $\operatorname{Hom}(L_i,Z(L_i))=1$, then $\operatorname{Aut}(L_1 \times ... \times L_t)=\operatorname{Aut}(L) \wr S_t$ where $L \cong L_i$. Quasisimple groups are perfect, and homomorphisms from perfect groups to abelian groups are trivial, implying that we do have $\operatorname{Hom}(L_i,Z(L_i))=1$ for each component $L_i$, and so $\operatorname{Aut}(L_1 \times ... \times L_t) \cong \operatorname{Aut}(L) \wr S_t$. Thus we have an injective homomorphism $\lambda: E \hookrightarrow \operatorname{Aut}(L) \wr S_t$. 
	
Let $\pi$ be the canonical projection $N_G(D,B_D) \to N_G(D,B_D)/C_G(D)$. Take $x \in N_G(D,B_D)$ such that $\pi(x)$ generates $E \cong C_{2^n-1}$, and let $\tilde{x}$ denote $\lambda \circ \pi (x) \in \operatorname{Aut}(L) \wr S_t$. Write $\tilde{x}=(\phi_x,\sigma_x)$, where $\phi_x=(\phi_{x_1},...,\phi_{x_t})$, $\phi_{x_i} \in \operatorname{Aut}(L_i)$, and $\sigma_x \in S_t$. Since $E$ acts transitively on the components by \Cref{E acts trans}, and $\phi_x$ fixes each component, $\sigma_x$ must be a cycle of length $t$, and the subgroups generated by each $\phi_{x_i}$ must act transitively on the non-trivial elements of $(C_2)^k \cong \Omega(D) \cap L_i$, for each $L_i$. Thus, $\phi_{x_i}$ must be a Singer cycle of $\operatorname{Aut}(\Omega(D) \cap L_i) \cong GL_k(2)$, i.e. $\phi_{x_i}$ has order $2^k-1$. 
	 
Since $\lambda$ is injective, this implies $\tilde{x}$ generates a subgroup of order $2^n-1$ in $C_{2^k-1} \rtimes C_t$. However, \Cref{tool E in wreath nope} tells us there is no such embedding when $k,t \geq 2$. If $k=1$, then each $b_i$ has cyclic defect group, and thus is nilpotent, contradicting \Cref{tool no nilp comps}. Therefore $t=1$. \end{proof}

\begin{lemma} \label{G/L solv}
Assume the conditions of \Cref{main}. If $B$ is reduced, then $G/F^*(G)$ and $G/L$ are solvable. 
\end{lemma}
\begin{proof}
 $C_G(F^*(G)) \leq F^*(G)$, so we can embed $G/F^*(G) \hookrightarrow \operatorname{Out}(F^*(G))$. \Cref{1comp} tells us that $E(G)=L$, and \Cref{O22} tells us that $O_2(G)=1$, so $F^*(G)=LZ(G)$. Since $G$ acts trivially on its centre, we have that $G/F^*(G) \leq \operatorname{Out}(L)$. Schreier's conjecture, which has been verified by the classification of finite simple groups, says that $\operatorname{Out}(L)$ is solvable, implying that $G/F^*(G)$ is solvable. Since $F^*(G)/L \leq Z(G)$, $G/L$ is also solvable. \end{proof}

%\begin{corollary} \label{2rank}
%Assume the conditions of \Cref{main}. Let $b$ be the $B$-covered block of $L$. If $B$ is reduced and $b$ is not nilpotent-covered, then the $2$-rank of $G/L$ is at most $3$. 
%\end{corollary}
%\begin{proof}
%Since $b$ has homocyclic defect group by \Cref{DcapN}, we must be in one of the six cases described in \Cref{ekkssimple}. By that result, if $L$ is one of $SL_2(2^n)$, ${}^2G_2(q)$, $J_1$, or $Co_3$, then the $2$-rank of $G/L$ is at most $1$. By assumption, $b$ is not nilpotent-covered, leaving only case the final case of \Cref{ekkssimple}, in which we have a decomposition $D \cap L = D_0 \times (C_2)^2$. Since $b$ is not nilpotent, \Cref{DcapN} tells us that $D \cap L \cong (C_{2^{m'}})^n$ for some $m' \leq m$. This group only admits such a decomposition when $m'=1$, implying $D \cap L$ is elementary abelian. Therefore, \Cref{ekkssimple} tells us that $L$ is of type $D_t(q)$ or $E_7(q)$ where $q$ is an odd prime power, and $t/2$ is odd. This implies $D_4(q)$ can't arise, and consulting \cite[Table 5]{atlas}, the outer automorphism groups of $D_t(q)$, ${}^2D_t(q)$, $E_6(q)$, and ${}^2E_6(q)$ have $2$-rank at most $3$.
%\end{proof}

Now we are ready to prove \Cref{main}:
\begin{proof}
Let $G$ be a finite group, and $B$ a block of $\mathcal{O}G$ with defect group $D \cong (C_{2^m})^n$ and inertial quotient $\mathbb{E}$ containing an element of order $2^n-1$. Then \Cref{intro singer maximal sub} tells us that $C_{2^n-1} \leq \mathbb{E} \leq C_{2^n-1} \rtimes C_n$. Thus $\mathbb{E}=E \rtimes F$ where $E \cong C_{2^n-1}$ and $F$ is isomorphic to an odd-order subgroup of $C_n$. \Cref{homo} tells us that $E$ acts freely on $D \backslash \{1\}$, and transitively on $\Omega(D) \backslash \{1\}$. 

%If $m=1$ and $n=|F|=3$, then $B$ is Morita equivalent to the principal block of one of the groups listed in \Cref{main} by the main theorem of \cite{eaton3}. Note that the Morita equivalences described in \Cref{allMor} between ${}^2G_2(q)$, $SL_2(8) \rtimes C_3$, and 

We suppose $B$ is a minimal counterexample: in particular, suppose $([G:O_{2'}(Z(G))],|G|)$ is minimised in the lexicographic ordering such that $B$ is not basic Morita equivalent to the principal block of any of $\mathcal{O}(SL_2(2^n) \rtimes F)$, $\mathcal{O}(D \rtimes \mathbb{E})$, or $\mathcal{O}J_1$.

We may assume $B$ is reduced in the manner described in \Cref{reduce remark}: suppose $N \lhd G$ and $b_N$ is a block of $\mathcal{O}N$ covered by $B$. \Cref{F1} tells us that there is a block $B_I$ of the stabiliser of $b_N$ in $G$ under conjugation that is Morita equivalent to $B$, with the same defect group and inertial quotient. By minimality, $B=B_I$. Applying this to every normal subgroup of $G$, $B$ is quasiprimitive. Further, by minimality and \Cref{F2 cor}, if $N \lhd G$ and $B$ covers a nilpotent block of $\mathcal{O}N$ then $N \leq Z(G)O_2(G)$.

If $D \lhd G$, then  \Cref{normal} tells us that $B$ is Puig equivalent to $\mathcal{O}(D \rtimes \mathbb{E})$, placing us in case (\ref{case1}) of \Cref{main} and contradicting our assumption that $B$ is a counterexample. Therefore $D$ is not normal in $G$, and so \Cref{O22} tells us that $O_2(G)=1$. This means that $Z(G)=Z(F^*(G))$, by \Cref{tool Z(G) < Z(F^*(G)) < O2(G)Z(G)}. Since $E$ acts freely on the non-trivial elements of $D$, \Cref{tool Mashke} implies that $O^2(G)=G$, i.e. $G$ has no normal subgroups of index $2$. \Cref{1comp} tells us that $G$ contains a single quasisimple component $L \lhd G$.

Since $B$ is reduced, $B$ covers a $G$-stable block $b$ of $L$. \Cref{tool no nilp comps} tells us $b$ is not nilpotent, and thus has a non-trivial defect group. Therefore, by \Cref{DcapN}, we have that $\Omega(D) \leq L$. Since $G/L$ is solvable by \Cref{G/L solv}, \Cref{solv quotient} tells us that $DL/L$ is a Sylow-$2$-subgroup of $G/L$. In fact, since $L$ cannot be $D_4(q)$ by \Cref{ekkssimple}, $\operatorname{Out}(L)$ is supersolvable by \cite[Table 5]{atlas}, and so $G/L \leq \operatorname{Out}(L)$ is also supersolvable. Thus, if $G/L$ is not odd, then there is some $L \lhd N \lhd G$ such that $[D:D \cap N]=2$, implying $B$ covers a block of $N$ with defect group $D \cap N \cong (C_{2^m})^{n-1} \times C_{2^{m-1}}$, contradicting \Cref{DcapN}. Therefore $b$ has defect group $D$ and $G/L$ is odd. \Cref{ekkssimple} lists the six possibilities for $b$ and $L$: \begin{itemize}
\item Suppose $L \cong SL_2(2^n)$ with $n \geq 2$. Then $m=1$, and $b$ has defect group $D \cong (C_2)^n$ and inertial quotient $E_b \cong C_{2^n-1}$. In particular, \Cref{outsl} tells us that $G \cong SL_2(2^n) \rtimes F$, placing us in case (\ref{case2}) of \Cref{main} and providing a contradiction.
	
\item Suppose $L \cong {}^2G_2(q)$, in which case $m=1$ and $n=|F|=3$. \cite[Proposition 3.1]{eaton3} tells us that $B$ is Puig equivalent to $b$, and by \Cref{allMor}, this implies that $B$ is Puig equivalent to $SL_2(8) \rtimes C_3$, as in case (\ref{case2}) of \Cref{main}, providing a contradiction. 

\item Suppose $L \cong Co_3$, in which case $m=1$ and $n=|F|=3$. Since $\operatorname{Out}(Co_3)=1$, we must have $G=L$ in this case, implying that $B$ is the unique non-principal block of $Co_3$. Thus \Cref{allMor} again tells us that $B$ is Puig equivalent to $SL_2(8) \rtimes C_3$, providing another contradiction. 

\item Suppose $L \cong J_1$, in which case $m=1$ and $n=|F|=3$. Since $\operatorname{Out}(J_1)=1$, we must have $G=L$. Thus $B$ is the principal block of $J_1$, placing us in case (\ref{case3}) of \Cref{main} and providing a contradiction.

\item Suppose $b$ is nilpotent-covered. Then \Cref{zhou} (\ref{zhou b nilp b inert}) tells us $b$ is inertial. Since $G/L$ is odd, \Cref{zhou} (\ref{zhou b inert B inert}) then tells us $B$ is inertial. That is, $B$ is basic Morita equivalent to $D \rtimes E$, the group in case (\ref{case1}) of \Cref{main}, contradicting our assumption.

\item Suppose $b$ lies in the final case of \Cref{ekkssimple}. Note that, in this case, there is a decomposition $D=(D \cap M_0) \times (C_{2})^2$. Such a decomposition is only possible if $D$ is elementary abelian, in which case  \Cref{ekkssimple} tells us that $L$ is of type $D_t(q)$ or $E_7(q)$, where $t/2$ is odd. Consulting \cite[Table 5]{atlas}, this implies that $G/L$ is contained in a cyclic group of odd order. Since $|E_b|=3$ is prime, we can apply \Cref{IQcor}, which tells us that there is a monomorphism $E_b \hookrightarrow \mathbb{E}$ whose image is normal. However, by \Cref{ekkssimple}, $E_b$ centralises $D \cap M_0$, while we showed in \Cref{homo Sing lem} that no normal subgroup of $\mathbb{E}$ centralises any non-trivial element of $D$. Thus we have a contradiction. 
\end{itemize}

Since \Cref{ekkssimple} lists every possibility for the structure of $b$ and $L$, and each of these yields a contradiction, there cannot be a minimal counterexample to \Cref{main}. Therefore the claim must hold. \end{proof}

\section*{Acknowledgements}
This paper is part of the work towards my PhD at the University of Manchester, which is supported by a University of Manchester Research Scholar award. I am deeply grateful to Charles Eaton, my PhD supervisor, for his crucial advice and patient support. I'm thankful to Cesare Ardito for countless productive discussions. Finally, I'm grateful to Gunter Malle for his careful readings of the preprint and helpful comments, and to Shigeo Koshitani for bringing to my attention the results and papers described in \Cref{allMor}, which improved the main result.

\newpage

\small
\bibliographystyle{amsplainab} \bibliography{biblio}

\end{document}